\documentclass[letterpaper,10pt]{amsart}

\usepackage[all]{xy}                      %
  
\CompileMatrices                            

\UseTips                                    

\input xypic
\usepackage[bookmarks=true]{hyperref}       

\usepackage{amssymb,latexsym,amsmath,amscd}
\usepackage{xspace}
\usepackage{color}
\usepackage{graphicx}
\usepackage{dsfont}

\reversemarginpar

\theoremstyle{plain}
\newtheorem{theorem}{Theorem}[section]
\newtheorem*{theorem*}{Theorem}

\newtheorem{lemma}[theorem]{Lemma}

\theoremstyle{definition}

\newtheorem{remark}[theorem]{Remark}

\DeclareMathOperator{\Res}{Res}
\DeclareMathOperator{\red}{red}

\newcommand{\enm}[1]{\ensuremath{#1}}          %

\newcommand{\cal}[1]{\mathcal{#1}}

\newcommand{\NN}{\enm{\mathbb{N}}}

\newcommand{\ZZ}{\enm{\mathbb{Z}}}

\newcommand{\PP}{\enm{\mathbb{P}}}

\newcommand{\KK}{\enm{\mathbb{K}}}

\newcommand{\Ii}{\enm{\cal{I}}}

\newcommand{\Oo}{\enm{\cal{O}}}

\renewcommand{\phi}{\varphi}
\renewcommand{\theta}{\vartheta}
\renewcommand{\epsilon}{\varepsilon}

\renewcommand{\to}[1][]{\xrightarrow{\ #1\ }}

\newcommand{\old}[1]{}

\begin{document}

\title[double lines]
{The Hilbert function of general unions of lines and double lines in the projective space}
\author{Edoardo Ballico}
\address{Dept. of Mathematics\\
 University of Trento\\
38123 Povo (TN), Italy}
\email{ballico@science.unitn.it}
\thanks{The author was partially supported by MIUR and GNSAGA of INdAM (Italy).}

\begin{abstract}
We study the Hilbert function of a general union $X\subset \PP^3$ of $x$ double lines and $y$ lines. In many cases (e.g.
always for $x=2$ and $y\ge 3$ or for  $x=3$ and $y\ge 2$
or for $x\ge 4$ and $y\ge \lceil(\binom{3x+4}{3} -27x-12)/(3x+2)\rceil +3-x$) we prove that $X$ has maximal rank. We give a few
examples of $x$ and $y$ for which $X$ has not maximal rank.
\end{abstract}

\maketitle

\section{Introduction}

Let $X\subset \PP^r$ be a closed subscheme. Fix $d\in \NN$. Recall that $X$ is said to have \emph{maximal rank in degree $d$} or \emph{maximal rank with respect to the line bundle $\Oo_{\PP^r}(d)$}  if  the
restriction map $H^0(\Oo_{\PP^r}(d))\to H^0(\Oo _X(d))$ has maximal rank as a linear map between finite-dimensional vector
spaces, i.e. it is either injective or surjective. Since $h^1(\Oo_{\PP^r}(d)) =0$,  $X$ has maximal rank in degree $d$ if and only if either $h^0(\Ii _X(d)) =0$ or $h^1(\Ii _X(d))=0$. The scheme $X$ is said to have \emph{maximal rank} if it has maximal rank in all degrees $d\in \NN$. Thus the scheme $X$ has maximal rank if and only if for each $d\in \NN$ either
$h^0(\Ii _X(d)) =0$ or $h^1(\Ii _X(d))=0$. In our case (general unions of lines and double lines), all integers
$h^0(\Oo_X(d))$, $d\in
\NN$, are known and hence knowing that $X$ has maximal rank means that
$h^0(\Ii _X(d)) =\max \{0,\binom{r+d}{r}-h^0(\Oo_X(d))\}$ for all $d$. 
 For any line
$L\subset
\PP^3$ let
$2L$ denote the closed subscheme of
$\PP^3$ with
$(\Ii _L)^2$ as its ideal sheaf. We say that $2L$ is a double line and that $L$ is the reduction of $2L$. The Hilbert
polynomial and the Hilbert function of $2L$ are well-known. Indeed, $h^0(\Oo _{2L}(d)) =3d+1$ for all $d$, $h^1(\Ii_{2L}(t))
=0$ for all $t\in \ZZ$ and $h^0(\Ii_{2L}(t)) =0$ if and only if $t\le 1$ (see Remark
\ref{16jun1}). Thus
$2L$ has maximal rank. R. Hartshorne and A. Hirschowitz proved that for all integers $r\ge 3$ and $b>0$ a general union
$X\subset \PP^r$ of $b$ lines has maximal rank. After several decades the problem was resurrected (\cite{a,bdsss,ccg1}). In \cite{a} it
was completely determined the Hilbert function of a general union $X\subset \PP^r$ of one double line and a prescribed number
of lines. In this paper we only consider the case
$r=3$. Any reader of \cite{a,hh} knows that this is from the numerical point of view the hardest and longest step and an
important step for the inductive proof of the general step. It has a key numerical simplification: we know exactly the Hilbert
polynomial of all schemes $X\subset \PP^3$ we need to prove to have maximal rank. We believe that we made a very efficient
use of it.

For all $(a,b)\in \NN^2\setminus \{(0,0)\}$ let $Z(a,b)$ denote the set of all schemes $X\subset \PP^3$ with $a+b$ connected
components, $a$ of them being double lines and $b$ of them being lines. The set $Z(a,b)$ is a smooth and irreducible
quasi-projective subset of the Hilbert scheme of $\PP^3$. Fix $X\in Z(a,b)$. Since $h^1(\Oo_X(t)) =0$ and $h^0(\Oo_X(t))
=a(3t+1)+b(t+1)$ for all
$t\in
\NN$ (Remark \ref{16jun1}), $X$ has maximal rank if and only if $h^0(\Ii_X(t)) =\max \{0,\binom{t+3}{3} -a(3t+1)-b(t+1)\}$ for
all $t\in
\NN$.

Let $A\subset \PP^3$ be any union of $a$ disjoint double lines. The
projective space $|\Ii _A(d)|$ parametrizes all degree $d$ surfaces of $\PP^3$ singular at all points of $A_{\red}$.
Thus if a general $X\in Z(a,b)$ has maximal rank with respect to the line bundle $\Oo_{\PP^3}(d)$, then we know the
dimension of the projective space parametrizing all degree $d$ surfaces containing $a+b$ general lines and singular at
all points of $a$ of them.

We prove the following results.
\begin{theorem}\label{i1}
A general $X_{2,b}\in Z(2,b)$ has maximal rank with respect to $\Oo _{\PP^3}(d)$, unless $b=2$ and $d=4$.
Moreover, $h^0(\Ii_{X_{2,2}}(4)) =1$ and $h^1(\Ii_{X_{2,2}}(4))=2$.
\end{theorem}

\begin{theorem}\label{i2}
A general $X_{3,b}\in Z(3,b)$ has maximal rank with respect to $\Oo _{\PP^3}(d)$, unless $(b,d)\in \{(0,4),(1,5)\}$. Moreover,
$h^0(\Ii _{X_{3,0}}(4)) =1$,
$h^1(\Ii_{X_{3,0}}(4))=5$,
$h^0(\Ii_{X_{3,1}}(5)) =4$, and $h^1(\Ii_{X_{3,1}}(5))=2$.
\end{theorem}

\begin{theorem}\label{i3}
Fix integers $a\ge 4$, $b\ge 0$ and $d\ge 3a+2$. A general $X\in Z(a,b)$ has maximal rank with respect to $\Oo_{\PP^3}(d)$.
\end{theorem}

\begin{theorem}\label{i4}
Fix integers $a\ge 4$ and $b\ge \lceil(\binom{3a+4}{3} -27a-12)/(3a+2)\rceil +3-a$. Then a general $X\in Z(a,b)$ has
maximal rank.
\end{theorem}

As far as we know the examples $X_{2,2}$, $X_{3,0}$ and $X_{3,1}$ were not known before. We got them trying to use the Horace
Lemma  for low degree surfaces and see when it fails (step (a) of the proof of Lemma \ref{c3}, Lemmas \ref{c2} and
\ref{c4}). We also prove that
$h^0(\Ii _{X_{4,0}}(6))=10$ and
$h^1(\Ii _{X_{4,0}}(6)) =2$ (Lemma \ref{t5}) and that $X_{4,1}$ has not maximal rank in degree $6$ (Remark \ref{set1}). To prove that they are exceptional is easy. Probably for $\PP^r$, $r>3$, a far
better strategy would be to use computers to detect potential cases and then trying to prove if they are exceptional.

Let $Z'(a,b)$ denote the set of all $X\subset \PP^3$ which are flat
limits of a family of elements of $Z(a,b)$. As in \cite{a,be,hh} we will use several times elements of $Z'(a,b)\setminus
Z(a,b)$ to prove statements for a general $X\in Z(a,b)$ using the semicontinuity theorem for cohomology.

We work over an algebraically closed field $\KK$ with characteristic $0$. We use this assumption to quote \cite{a, cm}.
Although the results proved in \cite{cm} are in general false in positive characteristic, the use of \cite{cm} (see
Remark \ref{cv03}) probably may be substituted with some hard work, numerical lemmas and intermediate lemmas, as was done in \cite{hh}.

We warmly thank a referee for very useful corrections and suggestions.

\section{Preliminaries and notation}

\begin{remark}\label{16jun1}
Let
$L\subset
\PP^3$,
be a line. Let
$2L$ denote the closed subscheme of
$\PP^3$ with
$(\Ii _L)^2$ as its ideal sheaf. Since the normal bundle $N_L$ of $L$ in $\PP^3$ is isomorphic to $\Oo_L(1)^{\oplus 2}$,
for each $t\in
\ZZ$ we have an exact sequence of coherent sheaves
\begin{equation}\label{eqa1}
0 \to \Ii_L(t-1)^{\oplus 2} \to \Ii_{2L}(t) \to \Ii _L(t)\to 0
\end{equation}
Thus $\chi (\Oo _{2L}(t)) = 3t+1$.
Fix integers $a\ge 0$, and $b\ge 0$ such that $(a,b)\ne (0,0)$.
Take any $X\in Z(a,b)$. Note that $\chi (\Oo _X(t)) = a(3t+1) +b(t+1)$.
From \eqref{eqa1} we get $h^1(\Oo _X(t))=0$,
$h^0(\Oo_X(t)) =a(3t+1)+b(t+1)$ for all $t\in \NN$ and $h^0(\Oo _X(t)) =0$ for all $t<0$. The exact
sequence $$0\to
\Ii_X(t)\to \Oo_{\PP^3}(t)\to \Oo_X(t)\to 0$$gives $h^2(\Ii_X(t)) =h^1(\Oo _X(t))=0$
for all $t\ge 0$, $h^1(\Ii_{2L}(t)) =0$ for all $t\in \ZZ$ and $h^0(\Ii _{2L}(t)) =0$ if and only if $t\le 1$. We often call
$X_{a,b}$ a general element of
$Z(a,b)$. 
\end{remark}

Let $M$ be a smooth quasi-projective variety. For any $p\in M$ let $(2p,M)$ denote the closed subscheme of $M$ with $(\Ii
_{p,M})^2$ as its ideal sheaf. We say that $(2p,M)$ is a double point of $M$. Note that $(2p,M)_{\red} =\{p\}$ and $\deg
((2p,M))=\dim M +1$. Now assume that $M$ is a linear subspace of $\PP^r$, $r\ge 1$.  Then $M$ is the minimal linear subspace
of $\PP^r$ containing $(2p,M)$. If $\dim M=2$ we say that $(2p,M)$ is a \emph{planar double point}. If $\dim M =1$ we say that
$(2p,M)$ is an \emph{arrow}. As in \cite{hh} we often use planar double points and arrows. If $M=\PP^3$ we often write $2p$
instead of $(2p,\PP^3)$. Thus $\deg (2p)=4$ and $2p$ spans $\PP^3$. Let $Q\subset \PP^3$ be a smooth quadric surface.
For any $o\in Q$ let $T_oQ$ denote the tangent plane of $Q$ at $o$. Note that $(2o,T_oQ)\subset Q$, i.e. $(2o,T_oQ)$ is the
closed subscheme of $Q$ with $(\Ii _{o,Q})^2$ as its ideal sheaf. We will say that $(2o,Q):= (2o,T_oQ)$ is the double point of $Q$ with $\{o\}$ as its reduction. Note that any arrow $v\subset (2o,T_oQ)$ is contained in $Q$. We will say that $v$ is an arrow
in $Q$ with $\{o\}$ as its reduction.

Fix a smooth quadric $Q\subset \PP^3$. For all $(a,u,v)\in \NN^3\setminus \{(0,0,0)\}$ let $W(a,u,v)$ denote the set of all
schemes $Y\subset \PP^3$ which have $a+u+v$ connected components, $a$ of them being double lines, $u$ of them being lines and
$v$ of them being reducible conics with singular point contained in $Q$. Note that $h^1(\Oo_Y(t)) =0$ and $h^0(\Oo _Y(t))
=a(3t+1) + u(t+1) +v(2t+1)$ for all $Y\in W(a,u,v)$. $W(a,u,v)$ is an irreducible subvariety of the Hilbert scheme
of $\PP^3$.  Elements in $W(a,u,v)$ are a key tool of our proofs. When we take $Y\in W(a,u,v)$ (resp. $X\in Z(a,u)$)  and
write $Y =A\cup U\cup V$ (resp.
$X =A\cup U$), then $A$ is the union of the double lines, $U$ is the union of the degree $1$ connected components and $V$ is the union of the reducible conics with singular point contained in $Q$.

For all $(a,d)\in \NN^2$ define the integers $b_{a,d}$, $c_{a,d}$, $u_{a,d}$ and $v_{a,d}$ by the relations
\begin{equation}\label{eqa1=}
a(3d+1) +(d+1)b_{a,d}+c_{a,d} =\binom{d+3}{3},\ 0 \le c_{a,d}\le d
\end{equation}
\begin{equation}\label{eqa2=}
a(3d+1) +(d+1)u_{a,d}-v_{a,d} =\binom{d+3}{3},\ 0 \le v_{a,d}\le d
\end{equation}

For $a\ge 0$, $e\ge 0$ and $d>0$ we define the following Assertions $C_e(a,d)$ and $C(a,d)$:

\quad {\bf Assertion} $C_e(a,d)$: We have $u_{a,d} \ge 2v_{a,d}+e$ and there is a curve $Y =A\cup U\cup V \in W(a,u_{a,d}-2v_{a,d},v_{a,d})$ such that $Q$ contains $e$ connected components of $U$ and $h^i(\Ii_Y(d))=0$, $i=0,1$.

\quad {\bf Assertion} $C(a,d)$: Set $C(a,d):= C_0(a,d)$.

By \eqref{eqa2=} $h^0(\Ii_Y(d)) =h^1(\Ii_Y(d))$ for any $Y\in W(a,u_{a,d}-2v_{a,d},v_{a,d})$.

Taking the difference between the equations in \eqref{eqa1=} for $(a,d)$ and in \eqref{eqa2=} for $(0,d-2)$ we get
\begin{equation}\label{eqa3=}
a(3d+1) +2u_{0,d-2}+(d+1)(b_{a,d}-u_{0,d-2})+c_{a,d} +v_{0,d-2}=(d+1)^2
\end{equation}

For all $a\ge 0$, $d\ge 2$ such that $a(3d+1) \le \binom{d+3}{3}$ the integer $b_{a,d}$ defined in \eqref{eqa1=} is
non-negative. For all such $a$, $d$ we define the following Assertions $E(a,d)$ and $F(a,d)$:

\quad {\bf Assertion} $E(a,d)$, $a\ge 0$, $d\ge 2$ and $a(3d+1)\le \binom{d+3}{3}$: A general $X\in Z(a,b_{a,d})$ satisfies
$h^1(\Ii_X(d)) =0$.  

\quad {\bf Assertion} $F(a,d)$, $a\ge 0$, $d\ge 3$ and $a(3d-2)\le \binom{d+2}{3}$: A general $X\in Z(a,b_{a,d-1}+1)$ satisfies
$h^0(\Ii_X(d-1)) =0$.

Take any $X\in Z(a,b_{a,d})$. By \eqref{eqa1=} $h^0(\Ii_X(d)) =c_{a,d}+h^1(\Ii_X(d))$.
Obviously if $h^1(\Ii_X(d))=0$, then a general $Y\in Z(a,b)$, $0\le b\le b_{a,d}$, satisfies $h^1(\Ii _Y(d))=0$.

\begin{remark}\label{2lu5}
Fix a general $X\in Z(a,b)$. If $a=0$, then $X$ has maximal rank (\cite{hh}) and in particular it has maximal rank with respect
to the line bundles $\Oo_{\PP^3}(1)$ and $\Oo_{\PP^3}(2)$. If $a>0$, then $h^0(\Ii_X(1))=0$. If $a=1$, then $X$ has maximal
rank with respect to the line bundle $\Oo_{\PP^3}(2)$. Now assume $a\ge 2$. Since any two skew lines span $\PP^3$ and the
singular locus of a quadric is a linear space, $h^0(\Ii_X(2)) =0$.
\end{remark}

\begin{remark}\label{3lu4}
For any integer $k>0$ we have $v_{0,3k+1} =v_{0,3k}=0$, $v_{0,3k-1}=2k$, $u_{0,3k-1} =(3k^2+3k+2)/2$, $u_{0,3k+1} =(3k+4)(k+1)/2$
and $u_{0,3k} = (k+1)(3k+2)/2$.
\end{remark}

\begin{remark}\label{cv03}
Let $W\subseteq H^0(\Oo_{\PP^3}(d))$ be any linear subspace. For any scheme $Z\subseteq \PP^3$ set $W(-Z):= W\cap
H^0(\Ii_Z(d))$. Fix an integer $e\ge 0$ and let $E\subset \PP^3$ be a general union of $e$ arrows of $\PP^3$. By \cite{cm}, we
have $\dim W(-E) =\max \{0,\dim W-2e\}$. We emphasize that we need that $E$ is general and in particular we need that
$E_{\red}$ is a general subset of $\PP^3$ with cardinality $e$. Fix a plane $H\subset \PP^3$. Let
$F\subset H$ be a general union of $e$ arrows in $H$. Applying the previous observation to the image of $W$ in
$H^0(\Oo_H(d))$ we get that either $\dim W(-F) =\max \{0,\dim W-2e\}$ or $\dim W(-H) \ge \dim W-2e+1$. Let $Q\subset \PP^3$ be
a smooth quadric. Fix $u\in \NN$ and $v\in \NN$. Let $W\subseteq H^0(\Oo_Q(u,v))$ be a linear subspace. Let $G\subset Q$ be a
general union of $e$ arrows in $Q$. Set $W(-G):= H^0(\Ii_G(u,v))\cap W$. By \cite{cm} we have $\dim W(-G) =\max \{0,\dim W
-2e\}$.
\end{remark}

\begin{remark}\label{cv04}
Let $T\subset \PP^3$,  be a reducible conic. Call $o$ the singular point of $T$. Set $E:= T\cup 2o$.
The scheme $E$ is a flat limit of a family of unions of $2$ skew lines of $\PP^3$ (\cite{ccg1}, \cite[Example 2.1.1]{hh}).
\end{remark}

\begin{remark}\label{c0}
Let $X\subset \PP^3$ be a closed subscheme. The residual scheme $\Res_Q(Y)$ of $Y$ with respect to the smooth quadric $Q$ is
the closed subscheme of $\PP^3$ with $\Ii_X:\Ii_Q$ as its ideal sheaf. The following exact sequence
\begin{equation}\label{eqc1}
0\to \Ii_{\Res_Q(X)}(t-2)\to \Ii_X(t)\to \Ii _{X\cap Q,Q}(t)\to 0
\end{equation}
is called the \emph{residual exact sequence with respect to $Q$}. The long cohomology exact sequence
of \eqref{eqc1} gives
$$h^i(\Ii_X(t))\le h^i(\Ii_{\Res_Q(X)}(t-2)) +h^i(Q,\Ii_{X\cap Q,Q}(t)), \ i\ge 0$$
We also obtain $h^0(\Ii _X(t)) \ge h^0(\Ii_{\Res_Q(X)}(t-2))$ with equality if $h^0(Q,\Ii_{X\cap Q,Q}(t))=0$.
For any plane $H\subset \PP^3$ the residual scheme $\Res_H(X)$ is the scheme with $\Ii_X:\Ii_H$ as its ideal sheaf. There is a
residual exact sequence with respect to $H$ similar to \eqref{eqc1} with $\Ii_{\Res_H(X)}(t-1)$ on the left.
\end{remark}

\begin{remark}\label{diff0}
Fix a smooth quadric $Q\subset \PP^3$, a scheme $X\subset \PP^3$ and positive integers $d$ and $s$. Let $S\subset
Q$ be a general subset of
$Q$ with cardinality
$s$. We have $h^0(\Ii _{X\cup S}(d)) =\min \{0,h^0(\Ii_X(d)) -s\}$ if $h^0(\Ii_{\Res_Q(X)}(d-2)) \le \min \{0,h^0(\Ii _X(d))
-s\}$. Thus if $h^1(\Ii_X(d)) =0$ and $h^0(\Ii _X(d))\ge h^0(\Ii_{\Res_Q(X)}(d-2))+s$, we have $h^1(\Ii _{X\cup S}(d))=0$. A
similar statement with $\Oo_{\PP^3}(d-1)$ instead of $\Oo_{\PP^3}(d-2)$ is true if instead of a quadric we take a plane.
\end{remark}

Fix  $(a,b)\in \NN^2$ such that either $a\ge 2$ or $a=1$ and $b>0$ or $b\ge 3$. The
\emph{critical value} of the pair $(a,b)$ or of any element of $Z'(a,b)$ is the minimal integer $d\ge 2$ such that $a(3d+1)
+b(d+1)\le \binom{d+3}{3}$. Call $1$ the critical value of $(1,0)$ and $(0,2)$.

\section{The main lemmas and the examples}

\begin{lemma}\label{d1}
Fix $(a,b)\in \NN^2$ such that either $a>0$ or $b\ge 2$. Fix any $X\in Z(a,b)$. Let $d$ be the critical value of the pair $(a,b)$.
$X$ has maximal rank if and only if $h^0(\Ii_X(d-1)) =0$ and $h^1(\Ii _X(d))=0$. 
\end{lemma}

\begin{proof}
Recall that $h^1(\Oo _X(x))=0$ and $h^0(\Oo _X(x)) = a(3x+1)+b(x+1)$ for all $x\ge 0$. Since $\dim X=1$ we have $h^i(\Oo _X(z))
=0$ for all $i\ge 2$ and all $z\in
\ZZ$. By the definition of critical value $X$ has maximal rank with respect to $\Oo _{\PP^3}(d)$ if and only if $h^1(\Ii
_X(d)) =0$. Assume $h^1(\Ii _X(d)) =0$. The exact sequence $$0 \to \Ii_X(x)\to \Oo_{\PP^3}(x)\to \Oo _X(x)\to 0$$ gives
$h^i(\Ii _X(d+1-i)) =0$ for all $i\ge 1$. The Castelnuovo-Mumford's lemma gives $h^i(\Ii _X(z)) =0$ for all $z>d$. Thus $X$
has maximal rank if
$h^0(\Ii _X(d-1)) =0$. To see that the condition $h^0(\Ii _X(d-1)) =0$ is a necessary condition for the maximal rank of $X$ it
would be sufficient to prove that
$a(3(d-1)+1) +bd \ge \binom{d+2}{3}$. If $d\ge 3$ this inequality is true by the definition of critical value. Note that every $Y\in Z(a,b)$ satisfies $h^0(\Ii _Y(1))
=0$ if $a>0$. If $a=0$ we use that that no plane contains two disjoint lines. Since $b\ge 2$, $h^0(\Ii _X(1)) =0$.
\end{proof}

\begin{remark}\label{3lu1}
Fix $(a,b)\in \NN^2\setminus \{0,0\}$ and a general $X\in Z(a,b)$.

\quad {\bf Claim 1:} $X$ has maximal rank with respect to $\Oo_{\PP^3}(3)$.

\quad {\bf Proof Claim 1:} If $a=0$ (resp. $a=1$), then Claim 1 is true by \cite{hh} (resp. \cite{a}). Assume $a\ge 2$. Since
$h^0(\Oo_X(3)) = 10a+4b$ and $\binom{6}{3} =20$, to prove Claim 1 it is sufficient to prove $h^0(\Ii_X(3)) =0$ if $a=2$ and
$b=0$. Assume $a=2$ and $b=0$. Fix two skew lines $L$ and $R$. Let $Q$ be a general quadric containing $L\cup R$. $Q$ is
smooth. Set $X:= 2L\cup 2R$. Since $h^0(Q,\Ii_{X\cap Q,Q}(3))=0$, $\Res _Q(X) = L\cup R$ and $h^0(\Ii _{L\cup R}(1))=0$, the
residual exact sequence of $Q$ gives $h^0(\Ii_X(3)) =0$.
\end{remark}

\begin{lemma}\label{3lu6}
Fix integers $a\ge 0$ and $d\ge 4$ such that $a(3d+1) \le \binom{d+3}{3}$ and $b_{a,d} \ge u_{0,d-2}-a$.
If $C_a(0,d-2)$ is true, then $E(a,d)$ is true.
\end{lemma}

\begin{proof}
Take a solution $(Y,Q)$ of $C_a(0,d-2)$ with $Y=U\cup V$ and $U =U'\cup U''$ with $U''\subset Q$, say with $U''\in |\Oo_Q(a,0)|$, and $V$ the union of the conics
with $\mathrm{Sing}(V)\subset Q$. We may assume that  $\mathrm{Sing}(V)$ is a general union of $v_{0,d-2}$ points of $Q$.
Set $G:=
\cup _{p\in
\mathrm{Sing}(V)}2p$. Let $E\subset Q$ be the union of $b_{a,d} - u_{0,d-2}+a$ general elements of $|\Oo_Q(1,0)|$. Set $X
=(\cup _{L\in U''}2L) \cup U'\cup V\cup G \cup E$. By Remark \ref{cv04} $X\in Z'(a,b_{a,d})$. By the semicontinuity theorem for
cohomology it is sufficient to prove $h^1(\Ii _X(d))=0$. Since $\Res_Q(X)=Y$ and $h^1(\Ii_Y(d-2)) =0$, the residual exact
sequence of $Q$ shows that it is sufficient to prove $h^1(Q,\Ii_{X\cap Q,Q}(d)) =0$, i.e. 
$$h^1(Q,\Ii _{(G\cap
Q)\cup (U'\cap Q)\cup (V\cap Q),Q}(d-b_{a,d}+u_{0,d-2}-3a,d)) =0.$$Set $Z:= (G\cap Q)\cup (U'\cap Q)\cup (V\cap Q)$. We only need to prove that $Z$ imposes independent conditions to the linear system $|\Oo_Q(d-b_{a,d}+u_{0,d-2}-3a,d)|$. The scheme $G\cap Q$ is a general union of
$v_{0,d-2}$  double points of $Q$ with $v_{0,d-2} \le d/3$ (Remark \ref{3lu4}). Thus $h^1(Q,\Ii_{Q\cap
G,Q}(d-b_{a,d}+u_{0,d-2}-3a,d))=0$ (one can use
\cite{laf,v}). Fix
$p\in \mathrm{Sing}(V)$ and let $D$ be the connected component of $V$ containing $p$. We may deform $D$ to a general union of
$2$ lines through $p$. Thus we may assume that $D\cap (Q\setminus \{p\})$ are two general points of $Q$. Thus we may fix
$G\cap Q$ and get that the other points of $Y\cap Q$ are general in $Q$.  By \eqref{eqa3=} $\deg (Z) =(d-b_{a,d}+u_{0,d-2}-3a+1)(d+1) -c_{a,d}$. Since  $c_{a,d}\ge 0$, we get $h^1(Q,\Ii _{(G\cap
Q)\cup (U'\cap Q)\cup (V\cap Q),Q}(d-b_{a,d}+u_{0,d-2}-3a,d)) =0$. 
\end{proof}

\begin{lemma}\label{3lu7}
Fix integers $a\ge 0$ and $d\ge 4$ such that $a(3d-2) \le \binom{d+2}{3}$ and $b_{a,d-1}+1 \ge u_{0,d-3}-a$.
If $C_a(0,d-3)$ is true, then $F(a,d)$ is true.
\end{lemma}

\begin{proof}
Take a solution $(Y,Q)$ of $C_a(0,d-3)$ with $Y=U\cup V$ and $U =U'\cup U''$, $U''\subset Q$, say with $U''\in
|\Oo_Q(a,0)|$, and $V$ the union of the conics with $\mathrm{Sing}(V)\subset Q$. We may assume that  $\mathrm{Sing}(V)$ is a
general union of $v_{0,d-3}$ points of $Q$.
Set $G:=
\cup _{p\in
\mathrm{Sing}(V)}2p$. Let $E\subset Q$ be the union of $b_{a,d-1}+1 - u_{0,d-3}+a$ general elements of $|\Oo_Q(1,0)|$. Set $X
=(\cup _{L\in U''}2L) \cup U'\cup V\cup G \cup E$. By Remark  \ref{cv04} $X\in Z'(a,b_{a,d-1})$. By the semicontinuity theorem
for cohomology it is sufficient to prove $h^0(\Ii _X(d-1))=0$. Since $\Res_Q(X)=Y$ and $h^0(\Ii_Y(d-3)) =0$, to conclude the
proof it is sufficient to prove $h^0(Q,\Ii_{X\cap Q,Q}(d-1)) =0$, i.e. $$h^0(Q,\Ii _{(G\cap Q)\cup (U'\cap Q)\cup (V\cap
Q),Q}(d-1-(b_{a,d-1}+1-u_{0,d-3}+3a),d-1)) =0.$$ Set $Z:= (G\cap Q)\cup (U'\cap Q)\cup (V\cap Q)$. Since $c_{a,d-3}\le
d$, We only need to prove  $h^0(\Ii_Z(d-1-(b_{a,d-1}+1-u_{0,d-3}+3a),d-1))=0$.   By \eqref{eqa3=} for the integer $d-1$ we have $\deg (Z) =(d-1-b_{a,d-1}+u_{0,d-3}-3a)d +d-c_{a,d-3}$. Recall that $c_{a,d-3}\le d$. The scheme $G\cap
Q$ is a general union of
$v_{0,d-3}$ double points of $Q$ and $v_{0,d-3} \le d/3$. We continue as in the proof of Lemma \ref{3lu6}.
\end{proof}

\begin{lemma}\label{n3lu1}
We have $a(3d+1) \le \binom{d+3}{3}$ and $b_{a,d}\ge u_{0,d-2}-a$ for all $d> 3a>0$.\end{lemma}

\begin{proof}
Call $\phi_a(d)$ the difference between the right hand side and the left hand side of the first inequality. Since for a fixed
$a\in \NN$ the function $\phi_a(x)$ is increasing for $x\ge 3a+1$, to prove the first inequality it is sufficient to observe
that $\phi_a(3a+1)\ge 0$.

 Assume that the second inequality fails, i.e. assume $b_{a,d}\le u_{0,d-2}-a-1$. Since $c_{a,d} \le d$,
\eqref{eqa3=} gives
\begin{equation}\label{eqn1}
a(3d+1) +2u_{0,d-2}-(d+1)(a+1) +d +v_{0,d-2} \ge (d+1)^2
\end{equation}

Assume for the moment $d\equiv 0,2\pmod{3}$  so that $v_{0,d-2} =0$ and $2u_{0,d-2} = (d+1)d/3$ (Remark \ref{3lu4}). In this case
\eqref{eqn1} gives
\begin{equation}\label{eqn2}
a(2d+1) \ge (2d^2+5d+7)/3
\end{equation}
For all $d\ge 3a$  \eqref{eqn2} is false. Now assume $d\equiv 1\pmod{3}$, say $d=3k+1$ with $k\in \NN$, so that $v_{0,d-2} =2k = 2(d-1)/3$
and $2u_{0,d-2} = 3k^2+3k+2 = (d+1)d/3 + 2/3$. Thus it is sufficient to disprove the inequality
\begin{equation}\label{eqn3}
a(2d+1) \ge (2d^2+3d+7)/3
\end{equation}
In this case it is sufficient to assume $d\ge 3a+1$.\end{proof}

\begin{lemma}\label{3lu3}
$C_e(0,3k-1)$ is true for all $k \ge 4$ and all $e\le 2k-5$.
\end{lemma}

\begin{proof}
 It is sufficient to do the case $e=2k-5$. Recall that $v_{0,3k-5} =v_{0,3k-3}= 0$, $v_{0,3k-1} = 2k$, $u_{0,3k-5} = (3k-2)(k-1)/2$, $u_{0,3k-3} =k(3k-1)/2$, $u_{0,3k-1} =(3k^2+3k+2)/2$ and hence $u_{0,3k-1}-u_{0,3k-3} =2k+1$ and $u_{0,3k-3}-u_{0,3k-5} =2k-1$. Take $(Y,Q)$ satisfying $C_e(0,3k-5)$ (Lemma \ref{c8}), say $Y = U'\sqcup U''$ with $\deg (U'') =2k-5$ and $U''\in |\Oo_Q(2k-5,0)|$. Take a general quadric $Q'\subset \PP^3$
and set $D:= Q\cap Q'$. For a general $(Q,Q')$ the scheme $D$ may be seen as a general element of $|\Oo_Q(2,2)|$ (and hence it is a linearly normal degree $4$ elliptic curve) and a general element of $|\Oo_{Q'}(2,2)|$. Let $E \subset Q'$ be a general union of $2k-1$ element of $|\Oo_{Q'}(1,0)|$. Set $Y':= Y\cup E\in Z(0,3k-3)$.

\quad {\bf Claim 1:} $h^i(\Ii_{Y'}(3k-3)) =0$, $i=0,1$.

\quad {\bf Proof of Claim 1:} Since $\Res_{Q'}(Y') =Y$ and $h^i(\Ii_Y(3k-5)) =0$, $i=0,1$, the residual exact sequence of $Q'$ shows that to prove Claim 1 it is sufficient to prove $h^i(Q',\Ii_{Q'\cap
Y',Q'}(3k-3))=0$, $i=0,1$. The scheme $Q\cap Y'$ is the union of $3$ algebraic sets:  $U'\cap Q'$, $E$ and  $U''\cap Q'$. The set $U''\cap Q'$ is formed by $2e$ points
of the elliptic curve $D\in |\Oo_{Q'}(2,2)|$.  Since we may deform $U''$ to general lines fixing $U'$, $Q$ and $Q'$, the scheme $U''\cap
Q'$ is formed by $2\deg (U'')$ general points of $Q'$, concluding the proof of Claim 1.

 Take a general
$D\in |\Oo _Q(2,2)|$. For a general
$U''$ each connected component of $U''$ meets $D$ at two points. Since $D$ is an elliptic curve and $\deg (\Oo_D(k-2,3k-3)) =
\Oo _{Q'}(2,2)\cdot \Oo_{Q'}(k-2,3k-3) =8k-10 >4k-10$,
$D\cap U''$ gives independent conditions to $H^0(\Oo_D(k-2,3k-3)$. Since $k\ge 4$, $h^1(\Oo_{Q'}(k-4,3k-5)) =0$ and hence the
exact sequence
$$0 \to \Oo_{Q'}(k-4,3k-5)\to \Oo_{Q'}(k-2,3k-3)\to \Oo_D(k-2,3k-3)\to 0$$gives the surjectivity of the restriction map $H^0(\Oo_{Q'}(k-2,3k-3))\to H^0(\Oo_D(k-2,3k-3))$. Thus
$h^1(Q',\Ii_{U''\cap Q',Q'}(k-2,3k-3))=0$. Since we may move $U'$ to general lines after fixing $U''$, $Q$ and $Q'$, the scheme $U''\cap Q'$ is a general subset of $Q'$, concluding the proof of Claim 1.

 We need a small modification of the construction just done. Since a general $p\in Q$ is contained in another quadric, we may
assume that $D$ contains a general $p\in Q$
 and hence without losing the condition $h^i(\Ii_Y(3k-5))=0$, $i=0,1$, we may assume that one of the point of $U'\cap Q$ is contained in $D$. Even with this additional
assumption the proof of Claim 1 works, because $\deg (\Oo_D(k-2,3k-5)) =8k-10 >4k-9$. Thus we may assume that $U'\cap Q$
contains exactly one point of $D$. 
 We deform $E$ to lines $E_1\nsubseteq Q'$, keeping fixed for each connected component $L$ of $E$ one of the two points of
$L\cap D$. Thus $Y_1:= Y\cup E_1$ contains $2k$ general points of $D$ while the other points of $Y_1\cap Q$ are general in $Q$.
 Let $F\subset Q'$ be the union of $2k+1$ elements of $|\Oo_{Q'}(1,0)|$, $2k$ of them containing a point of $Y_1\cap D$. As in the proof of Claim 1 we
 get $h^i(\Ii_{Y_1\cup F}(3k-1)) =0$, $i=0,1$. $Y_1\cup F$ satisfies $C_{2k-1}(0,3k-1)$, because all singular points of it are contained in $D\subset Q$ and $U''\subset Q$.\end{proof}

\begin{lemma}\label{c8}
Fix an integer $k\ge 2$. 

\quad (a) $C_e(0,3k)$ is true for all $e\le 2k+1$.

\quad (b) $C_e(0,3k+1)$ is true for all $e\le 2k+1$.

\quad ({c}) $C_e(0,3k-1)$ is true for all $e\le 1$.
\end{lemma}

\begin{proof}
We first check the numerical inequality in $C_e(0,d)$.

\quad {\bf Claim 1:} $u_{0,d} \ge 2v_{0,d} +e$ if $d=3k$ and $e\le 2k+1$ (resp. $d=3k+1$ and $e\le 2k+1$, resp. $d=3k-1$ and $e\le 1$).

\quad {\bf Proof of Claim 1:} Recall (Remark \ref{3lu4}) that $v_{0,3k} =v_{0,3k+1}=0$, $u_{0,3k}=(3k+2)(k+1)/2 \ge 2k+1$, $u_{0,3k+1}=(k+1)(3k+4)/2 \ge 2k+1$, $v_{0,3k-1}=2k$ and $u_{0,3k-1}=
(3k^2+3k+2)/2 =(k-1)(3k-2)/2 +2v_{0,3k-1}\ge 1+2v_{0,3k-1}$.

\quad {\bf Observation 1:} Since $v_{0,3k} =v_{0,3k+1}=0$, the main result of \cite{hh}, i.e. that general unions of a prescribed number of lines have maximal rank, gives $C(0,3k)$ and $C(0,3k+1)$
for all $k>0$.

In \cite[p. 177]{hh} R. Hartshorne and A. Hirschowitz defined the following Assertion $H'_{3k-1}$, $k\ge 2$:

\quad {\bf Assertion $H'_{3k-1}$, $k\ge 2$:} Let $Q\subset \PP^3$ be a smooth quadric. There is a scheme $U\cup V\subset \PP^3$ such that $U$ is the union of $(k-1)(3k-2)/2$ disjoint lines, $V$ is the union of $2k$ reducible conics with singular point contained in $Q$
and $h^i(\Ii_{U\cup V}(3k-1)) =0$, $i=0,1$.

\quad {\bf Observation 2:} $H'_{3k-1}$ is true for all $k\ge 2$ (\cite[Proposition 2.2]{hh}). Note that (except for the numerical inequality) $H'_{3k-1}$ is just $C(0,3k-1)$. Thus for all $k\ge 2$ we may use $C(0,3k)$, $C(0,3k+1)$ and $C(0,3k-1)$.

Fix a smooth quadric $Q\subset \PP^3$ and a ruling $|\Oo_Q(1,0)|$ of $Q$.

\quad{\bf Claim 2:} $C_e(0,3k)$ is true for all $k\ge
2$ and all $e\le 2k+1$.

\quad {\bf Proof of Claim 2:} Let $E\subset Q$ be
the union of $2k+1$ distinct elements of $|\Oo_Q(1,0)|$. Note that $u_{0,3k} =2k+1+u_{0,3k-2}$. Let
$T\subset
\PP^3$ be a general union of
$u_{0,3k-2}$ lines. Observation 1 gives $h^i(\Ii_T(3k-2)) =0$, $i=0,1$. Since $T$ is general, $T\cap E=\emptyset$ and hence to
prove Claim 2 it is sufficient to prove $h^i(\Ii_{T\cup E}(3k)) =0$. The residual exact sequence of $Q$ shows that it is
sufficient to prove $h^i(Q,\Ii _{Q\cap (E\cup T),Q}(3k))=0$, $i=0,1$, i.e. (since $E\cap T=\emptyset$) $h^i(Q,\Ii_{Q\cap
T,Q}(k-1,3k)) =0$. This is true, because $Q\cap T$ is a general union of $2u_{0,3k-2} = k(3k+1)$ points of $Q$.

\quad{\bf Claim 3:} $C_e(0,3k+1)$ is true for all $k\ge 2$ and all $e\le 2k+1$.

\quad {\bf Proof of Claim 3:} Note that $u_{0,3k+1}-u_{0,3k-1} =2k+1$. Let $E\subset Q$ be a union of $2k+1$
distinct elements of $|\Oo_Q(1,0)|$. Observation 2 gives the existence of
$U\cup V$, where
$U$ is a union of
$u_{0,3k-1}-2v_{0,3k-1} = (3k^2-5k+2)/2$ lines, $V$ is a union of $v_{0,3k-1}$ reducible conics, $\mathrm{Sing}(V)\subset Q$ and $h^i(\Ii_{U\cup
V}(3k-1)) =0$, $i=0,1$. Each reducible conic of $V$ intersects $Q$ in its singular point and 2 other points. For a general $U\cup V$ we may assume that $U\cap Q$ is a general union of
$2(u_{0,3k-1}-2v_{0,3k-1})+2v_{0,3k-1} = 3k^2-k+2$ points of $Q$, that $\mathrm{Sing}(V)$ is a general subset of $Q$ with cardinality
$v_{0,3k-1}$ and that $E\cap (U\cup V)=\emptyset$.
Set
$Z:=
\cup _{o\in \mathrm{Sing}(V)}2o$ and $T:= U\cup V\cup Z\cup E$. By Remark \ref{cv04} it is sufficient to prove
$h^i(\Ii_T(3k+1)) =0$,
$i=0,1$. Since $\Res_Q(T) =U\cup V$, the residual exact sequence of $Q$ shows that it is sufficient to prove
$h^i(Q,\Ii _{T\cap Q,Q}(3k+1)) =0$, i.e. (since $E\cap (U\cap V)=\emptyset$), $h^i(Q,\Ii_{(U\cup V\cup Z)\cap Q,Q}(k,3k+1))
=0$. Each connected component of $V\cup Z$ intersects $Q$ in a degree $5$ scheme (the union of two points and a double point of $Q$). Thus the scheme $(U\cup V\cup Z)\cap Q$ is a general union of $3k^2-k+2$ points of $Q$ and  $v_{0,3k-1} =2k$
double points $(2o,Q)$, $o\in \mathrm{Sing}(V)$, of $Q$. Since $\deg ((U\cup V\cup Z)\cap Q) =3k^2+5k+2 =(k+1)(3k+2)$ and
$U\cap Q$ is general, it is sufficient to prove $h^1(Q,\Ii _{A,Q}(k,3k+1))=0$, where $A$ is a general union of $2k$ double
points of $Q$. Use \cite{laf} or \cite{v}.

\quad {\bf Claim 4:} $C_e(0,3k-1)$ is true for all $k\ge 2$ and all $e\le 1$.

\quad {\bf Proof of Claim 4:} We have $v_{0,3k-1}=2k$. By $C(0,3k-3)$, $k\ge 2$,  a general union $U\subset
\PP^3$ of
$u_{0,3k-3}$ lines satisfies $h^i(\Ii_U(3k-3)) =0$ (Observation 1). The set $U\cap Q$ is formed by $(3k-1)k$ general points of $Q$. Note that
$u_{0,3k-1}-u_{0,3k-3} =2k+1$.  Let $F\subset Q$ be the union of $2k+1$ elements of $|\Oo_Q(1,0)|$, exactly $2k$ of them
containing a point of $U\cap Q$. Thus $U\cup F$ is a union of $u_{0,3k-1}-2v_{3k-1}$ lines, exactly one of them  contained in $Q$,
and $2k$ reducible conics whose singular point is contained in $Q$. Thus it is sufficient to prove $h^i(\Ii_{U\cup F}(3k-1))=0$, $i=0,1$.
Use first the residual exact sequence of $Q$, then the residual exact sequence of $F$ in $Q$ and that the $3k^2-3k$ points
$(U\cap Q)\setminus
U\cap F$ are general in
$Q$.\end{proof}

\begin{lemma}\label{c1}
$h^1(\Ii_{X_{2,0}}(t))=0$ for all $t\ge 3$ and $h^0(\Ii_{X_{2,0}}(3))=0$.
\end{lemma}

\begin{proof}
Set $X:= X_{2,0}$. Fix an integer $t\ge 3$. Let $Q$ be a smooth quadric containing $X_{\red}$, say as
an element of $|\Oo_Q(2,0)|$. Thus $X_{2,0}\cap Q\in |\Oo_Q(4,0)|$. Hence $h^1(Q,\Ii _{Q\cap X_{2,0},Q}(t)) =0$ for all $t\ge
3$.
Since $\Res_Q(X)$ is a union of $2$ disjoint lines, $h^1(\Ii_{\Res_Q(X)}(x))=0$ for all $x\ge 1$. Use the
residual exact sequence of $Q$ (Remark \ref{c0}). For $t=3$ we get $h^0(\Ii_X(3))=\binom{6}{3} -h^0(\Oo_X(3)) =0$.
\end{proof}

\begin{lemma}\label{c2}
$h^1(\Ii_{X_{3,0}}(t))=0$ for all $t\ge 5$, $h^0(\Ii_{X_{3,0}}(t)) =0$ for all $t\le 3$, $h^0(\Ii _{X_{3,0}}(4)) =1$ and
$h^1(\Ii_{X_{3,0}}(4))=5$.
\end{lemma}

\begin{proof}
Set $X:=X_{3,0}$. Let $Q$ be the only quadric containing $X_{\red}$. $Q$ is smooth. Call $|\Oo_Q(1,0)|$ the ruling of $Q$
such that $X\cap Q\in |\Oo_Q(6,0)|$. We have $h^1(Q,\Ii _{X\cap Q,Q}(t)) =0$ for all $t\ge 5$ and $h^0(Q,\Ii _{X\cap Q,Q}(t))
=0$ for all $t\le 5$. Thus $h^1(Q,\Ii _{X\cap Q,Q}(4)) =5$.
Since $\Res_Q(X)=X_{\red}$, $h^1(\Ii_{\Res_Q(X)}(t)) =0$ for all $t\ge 2$, $h^0(\Ii_{\Res_Q(X)}(2))
=1$ and  $h^0(\Ii_{\Res_Q(X)}(1))=0$. Use the residual exact sequence of $Q$.
\end{proof}

\begin{lemma}\label{c4}
$h^0(\Ii _{X_{3,1}}(5)) =4$, $h^1(\Ii _{X_{3,1}}(5)) =2$ and $h^0(\Ii_{X_{3,2}}(5)) =0$.
\end{lemma}

\begin{proof}
Set $X:= X_{3,1}$ and write $X =2L_1\cup 2L_2\cup 2L_3\cup L_4$ with $L_1\cup L_2\cup L_3\cup L_4$ a general union
of $4$ lines. Let $Q$ be the only quadric containing $L_1\cup L_2\cup L_3$. $Q$ is smooth. Note that $h^0(Q,\Ii
_{(2L_1,Q)\cup (2L_2,Q)\cup (2L_3,Q),Q}(5)) =0$. We have $\Res_Q(X) =L_1\cup L_2\cup L_3\cup L_4$. Thus $h^1(\Ii
_{\Res_Q(X)}(3)) =0$ and
$h^0(\Ii_{\Res_Q(X)}(3))=4$ (\cite{hh}). The residual exact sequence of $Q$ gives $h^0(\Ii_X(5))=4$. Since $h^0(\Oo_{\PP^3}(5))
=56$ and $h^0(\Oo _X(5)) =4$,
we get $h^1(\Ii _X(5))=2$. Take as $X_{3,2}$ the curve $X':= X\cup L$ with $L$ a general line. Since $\Res_Q(X')$ is a general union of $5$ lines, $h^i(\Ii_{\Res_Q(X')}(3)) =0$, $i=0,1$. Since $X'\cap Q$ is a general union of an element of $|\Oo_Q(6,0)|$ and $4$ points, $h^0(\Ii_{X'}(5)) =0$.
\end{proof}

\begin{lemma}\label{t5}
We have $h^0(\Ii _{X_{4,0}}(6))=10$ and $h^1(\Ii _{X_{4,0}}(6)) =2$.
\end{lemma}

\begin{proof}
Set $X:= X_{4,0}$ and write $X =2L_1\cup 2L_2\cup 2L_3\cup 2L_4$ with $X_{\red}=L_1\cup L_2\cup L_3\cup L_4$ a general union
of $4$ lines. Let $Q$ be the only quadric containing $L_1\cup L_2\cup L_3$. Note that $X\cap Q$ is the union of $3$
disjoint double lines of $Q$ and $2$ double points with $L_4\cap Q$ as its reduction. $Q$ is a smooth quadric and we
call $|\Oo_Q(1,0)|$ the pencil of lines of $Q$ containing $L_1$. Thus $\Ii _{(2L_1,Q)\cup (2L_2,Q)\cup (2L_3,Q),Q}(6,6)
\cong \Oo_Q(0,6)$. Thus $h^0(Q,\Ii _{X\cap Q,Q}(6,6)) = 3$ and $h^1(Q,\Ii _{X\cap Q,Q}(6,6)) = 2$. We
have
$\Res_Q(X) =L_1\cup L_2\cup L_3\cup 2L_4$. We have
$h^1(\Ii _{\Res_Q(X)}(4)) =0$ by \cite{a} and hence $h^0(\Ii _{\Res_Q(X)}(4)) =7$. The residual exact sequence of $Q$
gives $h^0(\Ii_X(6)) =10$ and hence $h^1(\Ii_X(6))=2$.
\end{proof}

\begin{remark}\label{set1}
Since $h^0(\Ii _{X_{4,1}}(6)) \ge h^0(\Ii _{X_{4,0}}(6)) -7$ and $h^1(\Ii _{X_{4,1}}(6)) \ge h^1(\Ii _{X_{4,0}}(6))$,
Lemma \ref{t5} gives $h^i(\Ii _{X_{4,1}}(6)) >0$ for $i=0,1$.
\end{remark}

\section{End of the proofs of the theorems}

\begin{lemma}\label{c3}
Theorem \ref{i1} is true for all $d\le 8$ and $F(2,9)$ is true.
\end{lemma}

\begin{proof}
Since $h^0(\Ii_{X_{2,0}}(3))=0$ (Lemma \ref{c1}), we may assume $d\in \{4,5,6,7,8\}$. Recall that
$h^1(\Ii_{X_{2,0}}(t)) =0$ for all
$t\ge 3$ (Lemma \ref{c1}). 

\quad (a) Take $d=4$. Since $\binom{7}{3} =35$ and $h^0(\Oo_{X_{2,b}}(4)) = 26+5b$, it is sufficient to prove
$h^1(\Ii_{X_{2,1}}(4))=0$
and $h^0(\Ii _{X_{2,2}}(4))=1$. Take a general line $G\subset \PP^3$ and set $T:= X_{2,0}\cup G$. Let $Q'$ be a smooth quadric
containing
$(X_{2,0})_{\red}$. We have
$h^1(Q,\Ii_{T\cap Q,Q}(4))=h^1(\Ii _{\Res_Q(T)}(2)) =0$ and hence $h^1(\Ii_{X_{2,1}}(4))=0$. Write $X_{2,2} =A\cup
R\cup L$ with $R$ and
$L$ lines. Let $Q$ be the only quadric containing $A_{\red}\cup R$. We have $h^0(Q,\Ii_{(A_{\red}\cup R)\cap Q,Q}(4)) =0$ and
hence $h^0(Q,\Ii_{X_{2,2}\cap Q,Q}(4)) =0$. Since $\Res_Q(X_{2,2})$ is the union of $3$ disjoint lines, $h^0(\Ii
_{\Res_Q(X_{2,2})}(2))=1$. Use Remark \ref{c0}.

\quad (b) Take $d=5$. Since $\binom{8}{3} =56$ and $h^0(\Oo_{X_{2,b}}(5)) = 32+6b$, it is sufficient to prove
$h^i(\Ii_{X_{2,4}}(5))=0$, $i=0,1$. Take general $L,R, D\in |\Oo_Q(1,0)|$. We take a general union $E\subset \PP^3$ of $3$ lines.
Set $X:= 2L\cup 2D\cup R\cup E$. Since $\Res_Q(X) =E\cup L\cup D$ is a general union of $5$ lines,
$h^i(\Ii_{\Res_Q(X)}(3))=0$, $i=0,1$ (\cite{hh}).
We have $h^i(Q,\Ii_{X\cap Q,Q}(5)) = h^i(Q,\Ii_{E\cap Q,Q)}(0,5)) =0$, $i=0,1$, because $E\cap Q$ is a general union of $6$
points. 

\quad({c}) Take $d=6$. Since $\binom{9}{3} =84$ and $h^0(\Oo_{X_{2,b}}(6)) = 38+7b$, it is sufficient to prove $h^0(\Oo _{X_{2,7}}(6))=0$
and $h^1(\Ii _{X_{2,6}}(6))=0$.  Note that $h^1(\Ii _{X_{2,6}}(6)) =0$ if and only if $h^0(\Ii _{X_{2,6}}(6)) =4$. Step (a) gives $h^1(\Ii _{X_{2,1}}(4)) =0$, i.e. $h^0(\Ii_{X_{2,1}}(4)) =4$. Take a general $Y\in Z(2,1)$. Thus $h^0(\Ii
_Y{X_{2,1}}(4)) =4$ and $Q\cap Y$ is a general union of $4$ double points of $Q$ and $2$ points of $Q$. Take a general union
$E\subset Q$ of $5$ elements of $|\Oo_Q(1,0)|$. Set $X:= Y\cup F$. To prove that $h^1(\Ii_{X_{2,6}}(6)) =0$ it is sufficient
to prove that
$h^1(\Ii_X(6)) =0$. By the residual exact sequences first of $Q$ and then  of $F$ as a divisor of $Q$ it is sufficient
to use that $h^i(Q,\Ii _{Y\cap Q,Q}(1,6))=0$, $i=0,1$ (\cite{laf,v}). Note that $h^0(\Oo _{X_{2,7}}(6))=0$ if and only if
$h^1(\Ii_{X_{2,7}}(6)) =3$.  Note that $h^0(\Ii _{X_{1,5}}(4)) =0$ and $h^1(\Ii _{X_{1,5}}(4)) =3$. Since any line is
contained in a smooth quadric, there is $W =A\cup U\in Z(1,5)$ such that $h^1(\Ii_W(4))=3$, $h^0(\Ii_W(4)) =0$ and $U =U'\cup
L$ with $L\in |\Oo_Q(1,0)|$. We may also assume that $A\cup U'$ is general, so that $(A\cup U')\cap Q$ is a general union of
$2$ double points of $Q$ and $8$ points of $Q$. Let $E\subset Q$ be a general union of $3$ elements of $|\Oo_Q(1,0)|$. Set
$T:= A\cup 2L\cup U'\cup E$. Note that $T\in Z(2,7)$,  $\Res_Q(T)=W$ and $(2L\cup E)\cap Q\in |\Oo_Q(5,0)|$. Use the
residual exact sequences first of $Q$, then of $(2L\cup E)\cap Q$ and then use \cite{laf,v}.

\quad (d) Take $d=7$. Since $\binom{10}{3} =120$ and $h^0(\Oo_{X_{2,b}}(7)) = 44+8b$, it is sufficient to prove $h^0(\Oo _{X_{2,10}}(7))=0$ and $h^1(\Ii _{X_{2,9}}(7))=0$. 
Set $Y:= X_{2,4}$. Let $E\subset Q$ be a general union of $6$ (resp. $5$) elements of $|\Oo_Q(1,0)|$. Set $X:= Y\cup E$. Since $\Res_Q(X)=Y$ and $h^i(\Ii _Y(5))=0$, $i=0,1$ (step (b)), it is sufficient to prove that $h^0(Q,\Ii _{Y\cap Q,Q}(1,7)) =0$ (resp. $h^1(Q,\Ii _{Y\cap Q,Q}(2,7)) =0$). Since $Y\cap Q$ is a general union of $4$ double points of $Q$
and $8$ points, it is sufficient to quote \cite{laf,v}.

\quad (e) Take $d=8$. Since $\binom{11}{3} = 165$ and $h^0(\Oo_{X_{2,b}}(8)) = 50+9b$, it is sufficient to prove $h^0(\Oo _{X_{2,13}}(8))=0$ and $h^1(\Ii _{X_{2,12}}(8))=0$.
To prove $h^1(\Ii _{X_{2,12}}(8))=0$ (resp. $h^0(\Oo _{X_{2,13}}(8))=0$) it is sufficient to use that $h^1(\Ii _{X_{2,6}}(6))=0$ (resp. $h^0(\Oo _{X_{2,7}}(6))=0$) and add $6$ general elements of $|\Oo_Q(1,0)|$. 
\end{proof}

\begin{proof}[Proof of Theorem \ref{i1}:]
Fix a general $X\in Z(2,b)$ and let $d$ be its critical value, i.e. the first integer $d\ge 2$ such that $b\le b_{2,d}$. By Lemma \ref{d1} it is sufficient to prove $h^0(\Ii _X(d-1)) =0$ and $h^1(\Ii _X(d)) =0$, i.e. it is sufficient to prove $F(2,d)$ and $E(2,d)$.
For $d\le 8$ use Lemma \ref{c3}, which also prove $F(2,9)$. For $d\ge 9$ we need to prove $E(2,d)$ and $F(2,d)$. To prove $E(2,d)$ (resp. $F(2,d)$)
it is sufficient to have $C_2(0,d-2)$ (resp. $C_2(0,d-3)$) and to test some numerical inequalities, which are satisfied because $a=2$ and hence $d-1>3a$ (apply Lemma \ref{n3lu1} to the integers $d$ and $d-1$).
To have $C_2(0,d-2)$ and $C_2(0,d-3)$ use Lemmas \ref{c8}, \ref{3lu3} and \ref{n3lu1}. The only missing case is $C_2(0,8)$, which is not covered by Lemma \ref{n3lu1}. Thus we need
to prove $E(2,10)$ and $F(2,11)$, i.e. (since $b_{2,10}= 20$ and $c_{2,10} =4$) we need to prove $h^1(\Ii _{X_{2,20}}(10)) =0$ and $h^0(\Ii_{X_{2,21}}(10))=0$. Recall that
$b_{2,8} =12$ and $c_{2,8} =7$. Thus adding $8$ general elements of $|\Oo_Q(1,0)|$ to $X_{2,9}$ we get $h^0(\Ii_{X_{2,21}}(10))=0$ as in step (d) of the proof of Lemma \ref{c3}. Since $c_{2,10}=4$,  $h^1(\Ii _{X_{2,20}}(10)) =0$ if and only if $h^0(\Ii _{X_{2,20}}(10)) =4$. We have $b_{0,8} =18$ and $c_{0,8} = 3\le  c_{2,10}$. By \cite{hh} $h^1(\Ii _{X_{0,8}}(8)) =0$. Since any two disjoint lines in $\PP^3$ are contained in a smooth quadric, with no loss of generality we may write $X_{0,8} =U'\cup U''$ with $U''$ the union of $2$ general
elements of $|\Oo_Q(1,0)|$. Take a general union $E\subset Q$ of $4$ elements of $|\Oo_Q(1,0)|$. Set $T:= (\cup_{L\in
U''}2L)\cup U'\cup E$. It is sufficient to prove that $h^1(\Ii _T(10)) =0$. Since $\Res_Q(T) =U$, it is sufficient to prove
that
$h^1(Q,\Ii _{T\cap Q,Q}(10)) =0$, i.e. that $h^1(Q,\Ii _{U'\cap Q,Q}(2,10))=0$, which is true because $U'\cap Q$ is a general union of $32$ points of $Q$.
\end{proof}

\begin{lemma}\label{6lu4}
Theorem \ref{i2} is true for all $d\le 11$  and $F(3,12)$ is true.
\end{lemma}

\begin{proof}
By Lemma \ref{c4} we may assume $d\ge 6$. Since $h^0(\Oo_{X_{3,0}}(d)) =3(3d+1)$, we have $b_{3,6} =3$, $c_{3,6}=6$, $b_{3,7} =6$, $c_{3,7}=6$, $b_{3,8}=10$, $c_{3,8}=0$, 
 $b_{3,9}=13$, $c_{3,9}=6$,  $b_{3,10}=17$, $c_{3,10}=6$,  $b_{3,11}=21$, $c_{3,11}=10$. First assume $d\equiv 0,2 \pmod{3}$,
i.e. assume $c_{0,d-2}=0$ (Remark
\ref{3lu4}). Let $U\subset \PP^3$ be a general union of $(d+1)d/6$ lines. Since any $3$ pairwise disjoint lines are contained
in a smooth quadric, we may assume $U =U'\cup U''$ with $U''$ union of $3$ general elements of $|\Oo_Q(1,0)|$ and $U'$ a
general union of $d(d+1)/6-3$ lines. We always have $b_{3,d}\ge b_{0,d-2}-3$. Let $E$ (resp. $F$) be a general union of
$b_{3,d}-b_{0,d-2}+3$ (resp. 
$b_{3,d}-b_{0,d}+4$) elements of $|\Oo_Q(1,0)|$. Set $X:= (\cup _{L\in U''}2L) \cup U'\cup E$ and $X':= (\cup _{L\in U''}2L) \cup
U'\cup F$. Since $\Res _Q(X)=\Res_Q(X')=U$, it is sufficient to use the residual exact sequence of $Q$. Now assume
$d\equiv 1\pmod{3}$, i.e. assume $d\in \{7,10\}$. We have $b_{1,5} =6$, $c_{1,5}=4$, $b_{1,8}=15$, $c_{1,8}=5$. We use the
previous proof using $X_{1,d-2}$. It works because $c_{1,d-2}\le c_{3,d}\le c_{1,d-2}+2$.
\end{proof}

\begin{proof}[Proof of Theorem \ref{i2}:]
Fix a general $X\in Z(3,b)$ and let $d$ be its critical value. By Lemma \ref{d1} it is sufficient to prove $h^0(\Ii _X(d-1)) =0$ and $h^1(\Ii _X(d)) =0$, i.e. it is sufficient to prove $F(3,d)$ and $E(3,d)$. By Lemma \ref{6lu4} we may assume $d\ge 12$. To prove $E(3,d)$ (resp. $F(3,d)$)
it is sufficient to have $C_3(0,d-2)$ (resp. $C_3(0,d-3)$) by Lemmas \ref{3lu6}, \ref{3lu7}  and to test some numerical inequalities, which are satisfied for $d\ge 11$ because $a=3$ and hence $d-1>3a$ (apply Lemma \ref{n3lu1} to the integers $d$ and $d-1$).
To have $C_3(0,d-2)$ and $C_3(0,d-3)$ use Lemmas \ref{c8} and \ref{n3lu1}. The line bundle $\Oo_{\PP^3}(d)$ with largest $d$  not covered in this way is $\Oo_{\PP^3}(13)$, i.e., since
$b_{3,13}=31$ and $c_{3,13} =6$, we need to prove $h^1(\Ii _{X_{3,31}}(13))=0$ and $h^0(\Ii_{X_{3,32}}(13)) =0$. We use that $c_{0,11} =4\le c_{3,13}$. Let $U\subset \PP^3$ be a general union of $b_{0,1}=30$ lines. Thus $h^1(\Ii _U(11))=0$ and $h^0(\Ii _U(11))=4$. Since any $3$ pairwise disjoint lines are contained in a smooth quadric, we may assume that $U =U'\cup L_1\cup L_2\cup L_3$, where $L_1,L_2,L_3$ are general elements of $|\Oo_Q(1,0)|$. Take a general union $E\subset Q$ of $4$ elements of $ |\Oo_Q(1,0)|$. Set $X:= 2L_1\cup 2L_2\cup 2L_3\cup U'\cup E$. Since $\Res_Q(X) =U$, the residual exact sequence of $Q$ shows that it is sufficient to prove $h^1(Q,\Ii _{U'\cap Q,Q}(3,13)) =0$. This is true, because $U'\cap Q$ is
formed by $54$ general points of $Q$. Take $X':= X\cup L\in Z(3,32)$, where $L\subset \PP^3$ is a general line. Since
$\Res_Q(X') =U\cup L$, $h^0(\Ii_{\Res_Q(X')}(11))=0$ (\cite{a}). Since $(U'\cup L)\cap Q$ is formed by $56$ general points
of $Q$, $h^0(Q,\Ii_{U'\cap Q}(3,13))=0$. Use the residual exact sequence of
$Q$.
\end{proof}

\begin{proof}[Proof of Theorem \ref{i3}:]
We need to prove that $h^0(\Ii _{X_{a,b}}(d)) =0$ if $b>b_{a,d}$ and $h^1(\Ii_{X_{a,b}}(d)) =0$ if $b\le b_{a,d}$. By Lemmas \ref{3lu6} and \ref{3lu7} it is sufficient to prove $C_a(0,d-2)$, $C_a(0,d-3)$ and to check a few numerical inequalities. The latter are true, because $d\ge 3a+2$ by assumption (Lemma \ref{n3lu1}). Suppose we need $C_a(0,3k+x)$
with $k\in \NN$ and $x\in \{-1,0,1\}$. For $x=0,1$ it is sufficient to assume $k\ge 3$ (Lemma \ref{c8}), which is true because $d-3\ge 10 =3\cdot 3+1$. For $x=-1$ it is sufficient to assume $k\ge 4$, which is satisfied if $d-3 \ge 11$.
The dependency on $a$ in these lemmas is hidden in the inequalities checked in Lemma \ref{n3lu1}.
\end{proof}

\begin{proof}[Proof of Theorem \ref{i4}:]
By Theorem \ref{i3} $X_{a,b}$ has maximal rank with respect to all line bundles $\Oo_{\PP^3}(d)$ with $d \ge 3a+2$. Thus it is sufficient to prove that $h^0(\Ii _{X_{a,b}}(3a+1)) =0$.
This is true because $h^0(\Ii _{X_{3,b+a-3}}(3a+1)) =0$ by Theorem \ref{i2} and our assumption $b+a-3 \ge \lceil(\binom{3a+4}{3} -3(9a+3+1))/(3a+2)\rceil$. 
\end{proof}

\end{document}